\documentclass[reqno,11pt]{amsart}
\usepackage{amssymb}
\usepackage{verbatim}
\newif\ifpdf
\ifpdf
  \usepackage[pdftex]{graphicx}
  \usepackage[pdftex]{hyperref}
\else
  \usepackage{graphicx}
\fi
\numberwithin{equation}{section}       
 \theoremstyle{plain}
 \newtheorem{thm}{Theorem}[section]
 \numberwithin{equation}{section} 
 \numberwithin{figure}{section} 
 \theoremstyle{plain}
 \theoremstyle{plain}
 \newtheorem{cor}[thm]{Corollary} 
 \theoremstyle{plain}
 \newtheorem{prop}[thm]{Proposition} 
 \theoremstyle{plain}
 \newtheorem{lem}[thm]{Lemma} 
 \theoremstyle{remark}
 \newtheorem{rem}[thm]{Remark}
 \theoremstyle{definition}

\theoremstyle{plain}
\newtheorem*{thmA}{Theorem A}
\theoremstyle{plain}

\theoremstyle{plain}
\newtheorem*{thmB}{Theorem B}
\theoremstyle{plain}

\theoremstyle{plain}

\theoremstyle{plain}

\theoremstyle{plain}

\theoremstyle{plain}

\theoremstyle{plain}
\newtheorem{defi}[thm]{Definition}

\newtheorem*{ackn}{Acknowledgements}

\newcommand{\A}{{\mathbb{A}}}
\newcommand{\B}{{\mathbb{B}}}
\newcommand{\C}{{\mathbb{C}}}
\newcommand{\N}{{\mathbb{N}}}
\newcommand{\PP}{{\mathbb{P}}}
\newcommand{\Q}{{\mathbb{Q}}}
\newcommand{\R}{{\mathbb{R}}}

\newcommand{\Z}{{\mathbb{Z}}}

\newcommand{\cA}{{\mathcal{A}}}

\newcommand{\cE}{{\mathcal{E}}}
\newcommand{\cF}{{\mathcal{F}}}
\newcommand{\cG}{{\mathcal{G}}}

\newcommand{\cL}{{\mathcal{L}}}

\newcommand{\cO}{{\mathcal{O}}}

\newcommand{\cV}{{\mathcal{V}}}

\newcommand{\cX}{{\mathcal{X}}}

\renewcommand{\a}{\alpha}

\newcommand{\de}{\delta}
\newcommand{\e}{\varepsilon}

\newcommand{\ord}{\operatorname{ord}}

\newcommand{\vol}{\operatorname{vol}}

\newcommand{\Spec}{\operatorname{Spec}}

\newcommand{\hdim}{\operatorname{\widehat{\dim}}}
\newcommand{\mas}{\mathrm{mass}}
\newcommand{\hvol}{\operatorname{\widehat{\vol}}}
\newcommand{\hdeg}{\operatorname{\widehat{\deg}}}
\newcommand{\bu}{\bullet}
\newcommand{\ma}{\mathrm{max}}
\newcommand{\mi}{\mathrm{min}}
\newcommand{\hd}{\widehat{\Delta}}
\newcommand{\hh}{\widehat{h}^0}
\newcommand{\al}{\mathrm{alg}}
\begin{document}

\setcounter{tocdepth}{1}

\title{Okounkov bodies of filtered linear series}

\date{\today}

\author{S{\'e}bastien Boucksom}

\address{CNRS, Institut de Math{\'e}matiques de Jussieu}

\email{boucksom@math.jussieu.fr}

\author{Huayi Chen}

\address{Universit{\'e} Paris Diderot -- Paris 7, Institut de Math{\'e}matiques de Jussieu}

\email{chenhuayi@math.jussieu.fr}

\begin{abstract} We associate to certain filtrations of a graded linear series of a big line bundle a concave function on the Okounkov body whose law with respect to Lebesgue's measure describes the asymptotic distribution of the jumps of the filtration. As a consequence we obtain a Fujita-type approximation theorem in this general filtered setting. We then specialize these results to the filtrations by minima in the usual context of Arakelov geometry, thereby obtaining in a simple way a natural construction of an arithmetic Okounkov body, the existence of the arithmetic volume as a limit and the arithmetic Fujita approximation theorem for adelically normed graded linear series.
\end{abstract}

\maketitle

\tableofcontents

\section*{Introduction}
\subsection{Okounkov bodies}
Let $X$ be an $n$-dimensional projective variety defined over an arbitrary field $K$ and let $L$ be a \emph{big} line bundle on $X/K$. Its \emph{Okounkov body}  $\Delta(L)\subset\R^n$ is a compact convex set designed to study the asymptotic behavior of the space of global sections $H^0(kL)$ as $k\to+\infty$ by generalizing to some extent the usual picture in toric geometry. Okounkov bodies were introduced and studied by Lazarsfeld and Musta\c{t}\v{a} in~\cite{LM08} and independently by Kaveh and Khovanskii in~\cite{KK08} (see also~\cite{KK09}) building on ideas of Okounkov~\cite{Ok96} (himself relying on former results of Khovanskii~\cite{Kho93}). They have the crucial property that
$$
\vol\Delta(L)=\lim_{k\to\infty}\frac{\dim_K H^0(kL)}{k^n}.
$$
Note that such a statement implicitely contains the \emph{existence} of the right-hand limit, a basic birational-geometric invariant of the big line bundle $L$ known (after multiplication by $n!$) as its \emph{volume}  $\vol(L)$.

It was more generally shown in~\cite{LM08} that one can attach to any \emph{graded linear series} $V_\bu$ of $L$ (i.e.~graded subalgebra of $R(L):=\bigoplus_{k\ge 0} H^0(kL)$) a convex body $\Delta(V_\bu)$ such that
$$
\vol\Delta(V_\bu)=\lim_{k\to\infty}\frac{\dim_K V_k}{k^n}
$$
as soon as $V_\bu$ \emph{contains an ample series} (cf.~Definition~\ref{defi:contample} below). The right-hand side is here again known (after multiplication by $n!$) as the \emph{volume}Ê $\vol(V_\bu)$ of $V_\bu$, and a general version of Fujita's approximation theorem in this setting was also obtained in~\cite{LM08}, to the effect that the volume of $V_\bu$ can be approximated by that of its \emph{finitely generated} graded subseries.

\subsection{Arakelov-geometric analogues}
Assume now that $K$ is a number field, and consider the following \emph{Arakelov-geometric} setting: let $\cX$ be a flat projective model of $X$ over the ring integers $\cO_K$, $\cL$ be a model of $L$ on $\cX$ and assume also given a conjugation-invariant continuous Hermitian metric on $L_\C$ over $X(\C)$, the whole data being summarized as $\overline L$. Given a finite set $S$ we set
$$
\hdim_K S:=\frac{1}{[K:\Q]}\log\sharp S.
$$
One then replaces $H^0(kL)$ with the finite set of \emph{small sections}
$$
\widehat{H}^0(k\overline L):=\{s\in H^0(\cX,k\cL),\,\sup_{X(\C)}|s|\le 1\},
$$
and $\dim_K H^0(kL)$ with $\hdim_K\widehat{H}^0(k\overline L)$, and it is a basic problem to construct the analogue of Okounkov bodies in this arithmetic setting (cf.~for instance~\cite{LM08} Question 7.7 P.51). This was largely accomplished by X.Yuan in~\cite{Yua08}, at least in the case of complete linear series. Indeed Yuan was able to construct a family of compact convex sets $\Delta_p(\overline L)\subset\R^{n+1}$ indexed by an infinite family of prime numbers $p$ in such a way that
$$\lim_{p\to\infty}\vol(\Delta_p(\overline L))\log p=\limsup_{k\to\infty}\frac{\hdim_K\widehat{H}^0(k\overline L)}{k^{n+1}}.
$$
The right-hand side (multiplied by $(n+1)!$) is known as the \emph{arithmetic volume} $\hvol(\overline L)$ of $\overline L$ (cf.~\cite{Mor09}), and Yuan inferred from the above result a Fujita-type approximation theorem for $\hvol(\overline L)$. Such a result reduces for many purposes the study of arithmetic volumes  to the \emph{arithmetically ample}  case, which has been well-understood since the work of Gillet-Soul\'e~\cite{GS92} (see also~\cite{AB95}) since it can be expressed as an arithmetic intersection number.

A Fujita-type approximation theorem was independently obtained by the second-named author in~\cite{Che08b} using a different approach.

\subsection{The general filtered setting}
Our goal in the present article is to unify in a natural and elementary way the above results. Let $V_\bu$ be as above a graded linear series of $L$ containing an ample series and consider the finite set
$$
\widehat{V}_k:=V_k\cap\widehat{H}^0(k\overline L).
$$
of small sections in $V_k$.
In a nutshell we will construct a (single) compact convex set $\hd(\overline V_\bu)\subset\R^{n+1}$
such that
\begin{equation}\label{equ:volhd}
\vol\hd(\overline V_\bu)=\lim_{k\to\infty}\frac{\hdim_K\widehat{V}_k}{k^{n+1}}
\end{equation}
and show that such an arithmetic volume can be approximated by that of finitely generated graded subseries.

More specifically, consider the \emph{filtration by minima}Ê $\cF$ of each $K$-vector space $V_k$, defined by letting for each $t\in\R$
$$
\cF_t V_k:=\text{Vect}_K\{s\in V_k\cap H^0(\cX,k\cL),\,\sup_{X(\C)}|s|\le e^{-t}\}.
$$
The \emph{jumping values}
$$
e_j(V_k,\cF):=\sup\{t\in\R,\,\dim\cF_t V_k\ge j\}\,\,(j=1,...,\dim V_k)
$$
of this filtration are then essentially equal to the classical \emph{successive minima}Ê of $\overline V_k$ and it thus follows from Gillet-Soul\'e's work~\cite{GS91} that
\begin{equation}\label{equ:gsintro}
\sum_{\begin{subarray}{c}1\leqslant j\leqslant\dim V_k\\e_j(V_k,\cF)>0\end{subarray}} e_j(V_k,\cF)=\hdim_K\widehat{V}_k+o(k^{n+1})
\end{equation}
On the other hand the filtration induced by $\cF$ on $V_\bu$ is  \emph{multiplicative} in the sense that
$$
(\cF_t V_k)(\cF_s V_m)\subset\cF_{t+s}V_{k+m},
$$
it is \emph{pointwise left-bounded}, i.e.~for each $k$ we have
$$
\cF_{-t} V_k=V_k\text{ for }t\gg 1
$$
and \emph{linearly right-bounded} in the sense that there exists $C>0$ such that
 $$
 \cF_{t} V_k=0\text{ for }t\ge Ck
 $$
(cf.~Lemma~\ref{lem:linbded}). We can then consider the Okounkov body
$$
\Delta(V_\bu^t)\subset\Delta(V_\bu)
$$
of the graded subseries $V_\bu^t$ defined by $V_k^t:=\cF_{kt}V_k$ and introduce
$$
G_{(V_\bu,\cF)}:\Delta(V_\bu)\to[-\infty,C]
$$
as the incidence function
$$
G_{(V_\bu,\cF)}(x):=\sup\{t\in\R,\,x\in\Delta(V_\bu^t)\}
$$
of the filtration $t\mapsto\Delta(V_\bu^t)$. We show that $G_{(V_\bu,\cF)}$ is concave, upper semicontinuous, and finite-valued on the interior $\Delta(V_\bu)^\circ$. Our main result is then the following.

\begin{thmA} Let $K$ be an arbitray field, $X$ be a projective $K$-variety and $L$ be a big line bundle on $X/K$. Let $V_\bu$ be a graded linear series of $L$ containing an ample series and let $\cF$ be a non-decreasing multiplicative $\R$-filtration of $V_\bu$ which is furthermore pointwise left-bounded and linearly right-bounded. Then the scaled jumping values $k^{-1}e_j(V_k,\cF)$ of the restriction of $\cF$ to $V_k$ equidistribute as $k\to\infty$ according to the law of $G_{(V_\bu,\cF)}$ with respect to Lebesgue measure.
\end{thmA}
In other words this result reads
$$
\lim_{k\to\infty}k^{-n}\sum_j f\left(k^{-1}e_j(V_k,\cF)\right)=\int_{\Delta(V_\bu)^\circ}\left(f\circ G_{(V_\bu,\cF)}\right)d\lambda
$$
for any bounded continuous function $f$ on $\R$, with $\lambda$ denoting Lebesgue measure on $\R^n$. This result extends in particular the second-named author's result~\cite{Che08b} Proposition 4.6,  expressing furthermore the limit measure as the push-forward of $\lambda$ by $G_{(V_\bu,\cF)}$.

We next define the \emph{filtered Okounkov body} of $(V_\bu,\cF)$ as the half-epigraph
$$
\hd(V_\bu,\cF):=\{(x,t)\in\Delta(V_\bu)\times\R,\,0\le t\le G_{(V_\bu,\cF)}(x)\}
$$
and set in the Arakelov-geometric case $\hd(\overline V_\bu):=\hd(V_\bu,\cF)$ with $\cF$ standing for the filtration by minima. Gillet-Soul\'e's result (\ref{equ:gsintro}) combined with Theorem A will then easily be seen to imply (\ref{equ:volhd}).

As a consequence of Theorem A we will also obtain the following arithmetic analogue of Lazarsfeld-Musta\c{t}\v{a}'s Fujita-type approximation theorem.
\begin{thmB} Let $X,L$ and $\cF$ be as in Theorem A. For each $\e>0$ there exists a finitely generated subseries $W_\bu$ of $V_\bu$ such that
$$
\vol\hd(W_\bu,\cF)\ge\vol\hd(V_\bu,\cF)-\e.
$$
\end{thmB}

\subsection{Relations to other works}
We have already mentioned above that Theorems A and B yield in particular simpler proofs of the main results of~\cite{Yua08,Che08b,Mor09}, and it is therefore conversely clear that these works have had a strong influence on the present article. The work of Witt-Nystr\" om~\cite{WN09} was also influential as far as the idea of constructing a concave function on the Okounkov body is concerned, even though the final outcome of our construction is not a Chebyshev-type function.

While the present work was being written X.Yuan has introduced in~\cite{Yua09} another kind of concave function $c[\overline L]$ on the Okounkov body of a line bundle endowed with Arakelov-geometric data. His construction is closer in spirit to that of~\cite{WN09}, since it consists in summing up the analogues of Witt-Nystr\"om's Chebyshev functions at all places of $K$. It is also closely related to Rumely-Lau-Varley's study of the \emph{sectional capacity}~\cite{RLV00}, and indeed Yuan's goal is to show that the mean-value of $c[\overline L]$ coincides with the quantity he denotes by $\vol_\chi(\overline L)$ and which is equal by definition to $-\log$ of the sectional capacity in~\cite{RLV00}.

We shall discuss in more details the relation between~\cite{Yua09} and the present work in Section~\ref{sec:yuan}. We will show in particular on an example that the two constructions do not coincide in general.

\subsection{Organization of the paper}
Let us briefly describe the structure of our article. 
\begin{itemize}
\item Section~\ref{sec:okounkov}  contains the necessary definitions and results on Okounkov bodies extracted from~\cite{LM08}, whereas Section~\ref{sec:filvec}Ê introduces some terminology related to filtrations. 

\item In Section~\ref{sec:conc} we define the concave function attached to a filtered graded linear series and prove our main results in this general setting. Theorem A corresponds to Theorem~\ref{thm:main} whereas Theorem B is Theorem~\ref{thm:fuj}.

\item We then show in Section~\ref{sec:arith} how to relate the Arakelov-geometric setting (and more generally the adelic setting) to the filtered case using filtrations by minima. 

\item In Section~\ref{sec:seccapa} we introduce a different filtration induced by arithmetic data, and show how to recover a slightly less general version of the main result of~\cite{RLV00} by applying our main result to this alternative filtration. 

\item The final Section~\ref{sec:yuan} discusses the relation between the present work and~\cite{RLV00,Yua09}.
\end{itemize}

\begin{ackn} The authors would like to thank Antoine Ducros, Charles Favre, Shu Kawaguchi, Vincent Maillot, Atsushi Moriwaki and David Witt-Nystr\"om for interesting conversations related to the present work. They are grateful to Xinyi Yuan for letter communications and helpful comments. The main result of the present article was presented by the second-named author at the R.I.M.S. Symposium on Hodge Theory and Algebraic Geometry. He would like to thank Atsushi Moriwaki and the organizers of the Symposium for the invitation.
\end{ackn}

\section{The concave transform of a filtered graded linear series}\label{sec:filter}
In this section $K$ denotes an arbitrary field.
\subsection{Okounkov bodies}\label{sec:okounkov}

The original idea of Okounkov bodies \cite{Ok96} was systematically developed in \cite{LM08} and independently in~\cite{KK08} in order to study the asymptotic behavior of graded linear series.

Let $X$ be a projective $K$-variety (i.e.~an integral projective $K$-scheme) of dimension $n$ and fix a system of parameters $z=(z_1,...,z_n)$ centered at a regular point $p\in X(\overline K)$ (with $\overline K$ denoting an algebraic closure of $K$). We then get a rank $n$ valuation
$$
\ord_z:\cO_X\setminus\{0\}\to\N^n
$$
centered at $p$ as follows: expand a given $f\in\cO_{X,p}$ as a formal power series
$$
f=\sum_{\a\in\N^n}a_\a z^\a
$$
and set
 $$
\ord_z(f)=\min_\mathrm{lex}\{\a\in\N^n,\,a_\a\neq 0\}
$$
where the min is taken with respect to the lexicographic order on $\N^n$. Note that $\nu$ depends on the choice of uniformizing parameters $z$ only \emph{via} the flag they induce in the tangent space at $p$ (compare~\cite{LM08} Section 5.2).

Given a line bundle $L$ on $X$ the function $\ord_z$ induces a valuation-like function on $H^0(L)\setminus\{0\}$ by composing it with the evaluation operator, with the basic property that the graded pieces have
\begin{equation}\label{equ:graded}
\dim\{s\in V,\,\ord_z(s)\ge_\mathrm{lex}\a\}/\{s\in V,\,\ord_z(s)>_\mathrm{lex}\a\}\le 1
\end{equation}
for each subspace $V\subset H^0(L)$. Indeed given $s_1,s_2$ with $\ord_z(s_j)=\a$, $i=1,2$, we have $s_j=c_j z^\a+\text{(higher order terms)}$ with $c_j\neq 0$ and it immediately follows that $s_1$ and $s_2$ are linearly dependent modulo $\{\ord_z>\a\}$ (see also~\cite{LM08} Lemma 1.3).

Now let $V_\bu$ be a \emph{graded linear series} of $L$, i.e. a graded $K$-subalgebra of
$$
R(L):=\bigoplus_k H^0(kL).
$$
One associates to $V_\bu$ the semigroup
$$
\Gamma(V_\bu):=\{(k,\a)\in\N^{n+1},\,\a=\ord_z(s)\text{ for some }s\in H^0(kL)\setminus\{0\}\}
$$
whose slice over $k$
$$
\Gamma_k:=\Gamma(V_\bu)\cap\left(\{k\}\times\N^n\right)
$$
has cardinality equal to $\dim V_k$ by (\ref{equ:graded}). The closed convex cone $\Sigma(V_\bu)\subset\R^{n+1}$ generated by $\Gamma(V_\bu)$ has a compact convex basis
$$
\Delta(V_\bu):=\Sigma(V_\bu)\cap\left(\{1\}\times\R^n\right)
$$
(cf.~\cite{LM08} P.18). By~\cite{LM08} Proposition 2.1 it is actually a convex body (i.e. has non-empty interior) and its Euclidian volume satisfies
\begin{equation}\label{equ:fuj}
\vol(\Delta(V_\bu))=\lim_{k\to\infty}k^{-n}\dim V_k
\end{equation}
as soon as $\Gamma(V_\bu)$ generates $\Z^{n+1}$ as a group. As shown in~\cite{LM08} Lemma 2.12 this is in particular the case when $V_\bu$ contains an ample series in the following sense:

\begin{defi}\label{defi:contample} Let $V_\bu\subset R(L)$ be a graded linear series of $L$. We will say that $V_\bu$ \emph{contains an ample series} if:
\begin{enumerate}
\item[(i)] $V_k\neq 0$ for all $k\gg 1$.
\item[(ii)] There exist a Kodaira-type decomposition $L=A+E$ into $\Q$-divisors with $A$ ample and $E$ effective such that
$$
H^0(kA)\subset V_k\subset H^0(kL)
$$
for all sufficiently large and divisible $k$.
\end{enumerate}
\end{defi}
This corresponds exactly to Condition (C) in~\cite{LM08} P.20, as explained in Remark 2.10 thereof.

\subsection{Filtered vector spaces and algebras}\label{sec:filvec}
By a filtration $\cF$ of a finite dimensional $K$-vector space $E$ we will always mean an $\R$-indexed non-increasing left-continuous filtration $t\mapsto\cF_t E$. Given a filtered vector space $(E,\cF)$ we set
$$
e_\mi(E,\cF):=\inf\{t\in\R,\,\cF_tE\neq E\}
$$
and
$$
e_\ma(E,\cF):=\sup\{t\in\R,\,\cF_t E\neq 0\}
$$
and we shall say that $\cF$ is left-bounded (resp.~right-bounded) if $e_\mi(E,\cF)>-\infty$ (resp.~$e_\ma(E,\cF)<+\infty$).

\begin{defi}\label{defi:jump} Let $\cF$ be a left and right-bounded filtration of an $N$-dimensional $K$-vector space $E$, $N\ge 1$.
\begin{itemize}
\item[(i)] The \emph{jumping values}
$$
e_\ma(E,\cF)=e_1(E,\cF)\ge...\ge e_N(E,\cF)=e_\mi(E,\cF)
$$
of $E$ are defined by
$$
e_j(E,\cF):=\sup\{t\in\R,\,\dim\cF_t E\ge j\}.
$$
\item[(ii)] The \emph{mass} of $(E,\cF)$ is defined by
$$
\mas(E,\cF):=\sum_{e_j(E,\cF)>0} e_j(E,\cF)
$$
(The sum over an empty set being equal to $0$ as usual).
\end{itemize}
\end{defi}
We thus have
$$
\dim\cF_t E=j\Longleftrightarrow t\in ]e_{j+1}(E,\cF),e_j(E,\cF)]
$$
(with $e_{N+1}(E,\cF)=-\infty$ and $e_0(E,\cF)=+\infty$ by convention). We note for future use that
\begin{equation}\label{equ:distrib}
\frac{d}{dt}\dim\cF_ tE=-\sum_{j=1}^N\de_{e_j(E,\cF)}
\end{equation}
holds in the sense of distributions.

In what follows a graded $K$-algebra $V_\bu=\bigoplus_{k\in\N} V_k$ will always be indexed by $\N$ and with finite dimensional pieces such that $V_0=K$.

\begin{defi} Let $\cF$ be an $\R$-indexed, non-decreasing, left-continuous filtration on a graded $K$-algebra $V_\bu$. We shall say that
\begin{itemize}
\item[(i)] $\cF$ is \emph{multiplicative} if
$$
(\cF_ t V_k)(\cF_s V_m)\subset\cF_{t+s} V_{k+m}
$$
holds for all $k,m\in\N$ and $s,t\in\R$.
\item[(iii)] $\cF$ is \emph{pointwise left-bounded (resp.~pointwise right-bounded)} if $(V_k,\cF)$ is left-bounded (resp.~right-bounded) for each $k$.
\item[(ii)] $\cF$ is \emph{linearly left-bounded (resp.~linearly right-bounded)} if there exists $C>0$ such that $e_\mi(V_k,\cF)\ge -Ck$ (resp.~$e_\ma(V_k,\cF)\le C k$).
\end{itemize}
\end{defi}

If $\cF$ is a multiplicative filtration on $V_\bu$ then setting
$$
V_k^t:=\cF_{kt}V_k
$$
defines a graded subalgebra $V_\bu^t:=\bigoplus_k V_k^t$
of $V_\bu$.

We introduce
\begin{equation}\label{equ:emin}
e_\mi(V_\bu,\cF):=\liminf_{k\to\infty}\frac{e_\mi(V_k,\cF)}{k}
\end{equation}
and

\begin{equation}\label{equ:emax}
e_\ma(V_\bu,\cF):=\limsup_{k\to\infty}\frac{e_\ma(V_k,\cF)}{k}
\end{equation}
Note that $\cF$ is linearly left (resp.~right)-bounded iff $e_\mi(V_\bu,\cF)$ (resp.~$e_\ma(V_\bu,\cF)$) is finite.

\begin{lem}\label{lem:subadd} Assume that the graded algebra $V_\bu$ is integral and satisfies $V_k\neq 0$ for all $k\gg 1$. Then we have
$$
e_\ma(V_\bu,\cF):=\lim_{k\to\infty}\frac{e_\ma(V_k,\cF)}{k}=\sup_{k\ge 1}\frac{e_\ma(V_k,\cF)}{k}.
$$
\end{lem}
\begin{proof} It is easily checked under the standing assumptions that $e_\ma(V_k,\cF)$ is super-additive in $k$ for $k\gg 1$, and the result follows from a standard result sometimes known as ``Fekete's lemma''.
\end{proof}

\begin{lem}\label{lem:fingen} If the graded algebra $V_\bu$ is finitely generated then any multiplicative filtration on $V_\bu$ which is pointwise left-bounded is linearly left-bounded.
\end{lem}
\begin{proof} Let $p$ such that $V_\bu$ is generated by $V_1+...+V_p$ and choose $t$ such that $\cF_{jt} V_j=V_j$ for $j=1,...,p$. Writing a given $x\in V_k$ as a homogeneous polynomial in elements of $V_1,...,V_p$ and using the mutliplicative property of $\cF$ it is then easily seen that $x$ lies in $\cF_{kt}V_k$.
\end{proof}

\subsection{The concave function of a filtered algebra}\label{sec:conc}
In what follows we fix a projective $K$-variety $X$ of dimension $n$, a set of uniformizing parameters $z$ at a given regular $\overline{K}$-point and a big bundle $L$ on $X$.

\begin{lem}\label{lem:contample} Let $\cF$ be a multiplicative filtration on a graded linear series $V_\bu$ of $L$ containing an ample series. For each $t<e_\ma(V_\bu,\cF)$ the graded linear series $V_\bu^t$ also contains an ample series.
\end{lem}
This gives in particular an elementary proof of~\cite{Che08b} Proposition 4.9.
\begin{proof}
Given $t< e_{\max}(V_\bu,\cF)$ we have $kt<e_\ma(V_k,\cF)$ for each $k\gg 1$ thus $V_k^t=\cF_{kt} V_k\neq 0$ by definition, and we see that (i) of Definition~\ref{defi:contample} is satisfied.

Let us now turn to (ii). By Definition~\ref{defi:contample} there exists an integer $m$, an ample line bundle $A$ and a non-zero section $s\in H^0(mL-A)$ such that the image $W_k$ of the map $H^0(kA)\to H^0(km L)$ given by multiplication by $s^k$ is contained in $V_{km}$ for all $k\ge 1$. Since $W_\bu$ is finitely generated there exists $a\in\R$ such that $W_k\subset V_{km}^a$ for all $k$ by Lemma~\ref{lem:fingen}.

On the other hand let $\e>0$ such that $t+\e<e_{\max}(V_\bu,\cF)$, and for each $p\gg 1$ pick  a non-zero element $v_p\in\cF_{t+\e}V_p$. The image of map
$$
H^0(kA)\to V_{k(m+p)}
$$
given by mutliplication by $(s v_p)^k$ lies in
$$
W_k\cdot V_{kp}^{t+\e}\subset V_{km}^a V_{kp}^{t+\e}\subset V_{k(m+p)}^{s_p}
$$
with
$$
s_p:=\frac{akm+(t+\e)kp}{km+kp}
$$
Now we have $s_p>t$ for $p\gg 1$, and it follows that the image is contained in $V_{k(m+p)}^t$ as desired.
\end{proof}

The Okounkov bodies $\Delta(V_\bu^t)$ make up a non-increasing family of compact convex subsets of $\Delta(V_\bu)$. If $\cF$ is multiplicative then it is straightforward to check that the following  \emph{convexity property} holds:
\begin{equation}\label{equ:conv}
t\Delta(V_\bu^{s_1})+(1-t)\Delta(V_\bu^{s_2})\subset\Delta(V_\bu^{ts_1+(1-t)s_2})
\end{equation}
for all $s_1,s_2\in\R$ and $0\le t\le 1$.

On the other hand if $\cF$ is pointwise left-bounded then for each $k$ we have $V_k=\bigcup_{t\in\R}V_k^t$, which implies
$$
\Gamma(V_\bu)=\bigcup_{t\in\R}\Gamma(V_\bu^t).
$$
By Lemma~\ref{lem:contample} it follows as recalled above that each of the semigroups $\Gamma(V_\bu)$ and $\Gamma(V_\bu^t)$, $t<e_\ma(V_\bu,\cF)$ spans $\Z^{n+1}$ (\cite{LM08} Lemma 2.12). By~\cite{LM08} Lemma A.3 we thus get
\begin{equation}\label{equ:union}
\Delta(V_\bu)^\circ=\bigcup_{t\in\R}\Delta(V_\bu^t)^\circ
\end{equation}
when $\cF$ is pointwise left-bounded.

\begin{defi}\label{defi:conc} Let $\cF$ be a multiplicative filtration of a graded linear series $V_\bu$ of $L$ containing an ample series, and assume that $\cF$ is pointwise left-bounded and linearly right-bounded.
\begin{enumerate}
\item[(i)] The \emph{concave transform} of $\cF$ is
the concave function
$$
G_{(V_\bu,\cF)}:\Delta(V_\bu)\to[-\infty,+\infty[
$$
defined by
$$
G_{(V_\bu,\cF)}(x):=\sup\{t\in\R,\,x\in\Delta(V_\bu^t)\}.
$$
\item[(ii)] The \emph{filtered Okounkov body} is defined as
$$
\hd(V_\bu,\cF):=\{(x,t)\in\Delta(V_\bu)\times\R,\,0\le t\le G_{(V_\bu,\cF)}(x)\}
$$
\end{enumerate}
\end{defi}
The function $G_{(V_\bu,\cF)}$ is indeed concave by (\ref{equ:conv}). It takes finite values on the interior $\Delta(V_\bu)^\circ$ by (\ref{equ:union}) and is thus continuous there since it is concave. It is furthermore upper semi-continuous on the whole of $\Delta(V_\bu)$ since we have for each $t\in\R$
\begin{equation}\label{equ:sublevel}
\{\,G_{(V_\bu,\cF)}\ge t\}=\bigcap_{s<t}\Delta(V_\bu^s).
\end{equation}
Note also that
\begin{equation}\label{equ:bounds}
e_\mi(V_\bu,\cF)\le G_{(V_\bu,\cF)}\le e_\ma(V_\bu,\cF)
\end{equation}
on $\Delta(V_\bu)$.

\begin{rem}\label{rem:wn} $G_{(V_\bu,\cF)}$ can also be obtained as the \emph{concave envelope} in the sense of~\cite{WN09} Section 3 of the super-additive function $g:\Gamma(V_\bu)\to\R$ defined by
\begin{equation}\label{equ:g}
g(k,\a):=\sup\{t\in\R,\,\a=\ord_z(s)\text{ for some }s\in\cF_t V_k\}.
\end{equation}
\end{rem}

We are now in a position to state our main results.

\begin{thm}\label{thm:main}
Let $\cF$ be a pointwise left-bounded and linearly right-bounded multiplicative filtration of a graded linear series $V_\bu$ of $L$ containing an ample series. Then the measures on $\R$ defined by
$$
\mu_k:=k^{-n}\sum_j\delta_{k^{-1}e_j(V_k,\cF)}.
$$
converge weakly to $\left(G_{(V_\bu,\cF)}\right)_*\lambda$ as $k\to\infty$, where $\lambda$ denotes the restriction of the Lebesgue measure of $\R^n$ to $\Delta(V_\bu)^\circ$.
\end{thm}

\begin{proof} It is equivalent to show the convergence in the sense of distributions. By (\ref{equ:distrib}) the measure $-\mu_k$ is the distributional derivative of the non-increasing piecewise constant function
$$
h_k(t):=k^{-n}\dim\cF_{kt }V_k=k^{-n}\dim V_k^t.
$$
On the other hand $-\left(G_{(V_\bu,\cF)}\right)_*\lambda$ is the distributional derivative of
$$
h(t):=\vol\{G_{(V_\bu,\cF)}\ge t\}
$$
which is the left-continuous modification of the non-increasing function $g(t):=\vol(\Delta(V_\bu^t))$ by (\ref{equ:sublevel}) and monotone convergence. Now Lemma~\ref{lem:contample} combined with (\ref{equ:fuj}) implies that
$h_k(t)\to g(t)$ for all $t<e_\ma(V_\bu)$, thus $h_k\to h$ a.e. Since
$$
0\le h_k(t)\le k^{-n}\dim H^0(kL)
$$
is uniformly bounded, it follows by dominated convergence that $h_k\to h$ in $L^1_{loc}(\R)$ and we get $h_k'\to h'$ as distributions as desired.
\end{proof}

\begin{rem}\begin{itemize}
\item[(i)] The convergence in Theorem~\ref{thm:main} holds in the weak topology, i.e.~against all continuous functions with compact support. Note however that the total mass of $\mu_k$ is equal to $k^{-n}\dim V_k$, which converges to $\vol(\Delta(\Gamma))$, the total mass of $\left(G_{(V_\bu,\cF)}\right)_*\lambda$. Indeed this follows from (\ref{equ:fuj}) since $V_\bu$ is assumed to contain an ample series. It follows from a standard result in integration theory that the convergence holds as well against all \emph{bounded}  continuous functions on $\R$, and the result can be reformulated in more probabilistic terms by saying that the scaled jumping values
$$
k^{-1}e_1(V_k,\cF)\ge ...\ge k^{-1} e_{N_k}(V_k,\cF)
$$
of the sequence of filtered vector spaces $(V_k,\cF)$ equidistribute as $k\to\infty$ according to the law of $G_{(V_\bu,\cF)}$ (with respect to the Lebesgue measure on $\Delta(V_\bu)$ scaled to mass $1$). This is the formulation we adopted in Theorem A in introduction.

\item[(ii)] We also remark that Theorem~\ref{thm:main}  shows in particular that the law, or pushed-forward measure, $\left(G_{(V_\bu,\cF)}\right)_*\lambda$ does not depend on the choice of the regular system of parameters $z$, even though the Okounkov body $\Delta(V_\bu)$ and hence the function $G_{(V_\bu,\cF)}$ will depend on $z$ (or rather on the infinitesimal flag it defines) in general.
\end{itemize}
\end{rem}

\begin{cor}\label{cor:hdim} With the same assumptions as in Theorem~\ref{thm:main} the euclidian volume of the filtered Okounkov body satisfies
$$
\vol(\hd(V_\bu,\cF))=\int_{t=0}^{+\infty}\vol(\Delta(V_\bu^t))dt=\lim_{k\to\infty}\frac{\mas(V_k,\cF)}{k^{n+1}}.
$$
\end{cor}
\begin{proof} Let $e:=e_\ma(V_\bu,\cF)$. In the notation of Theorem~\ref{thm:main} we have
$$
\frac{\mas(V_k,\cF)}{k^{n+1}}=\int_0^{+\infty} t \mu_k(dt)
$$
and
$$
\vol(\hd(V_\bu,\cF))=\int_0^{+\infty} t\left(G_{(V_\bu,\cF)}\right)_*\lambda(dt)
$$
Now both $\mu_k$ and $\left(G_{(V_\bu,\cF)}\right)_*\lambda$ are supported on $]-\infty,e]$ by Lemma~\ref{lem:subadd} and the upper bound $G_{(V_\bu,\cF)}\le e$. The result thus follows by applyiny Theorem~\ref{thm:main} to a continuous function with compact support that coincides with $\max(t,0)$ on $]-\infty,e]$.
\end{proof}

We finally get a filtered version of the Lazarsfeld-Musta\c{t}\v{a} approximation theorem.

\begin{thm}\label{thm:fuj} Let $\cF$ and $V_\bu$ be as in Theorem~\ref{thm:main}. Then for each $\e>0$ there exists a finitely generated graded subalgebra $W_\bu$ of $V_\bu$ such that
$$
\vol\hd(W_\bu,\cF)\ge\vol\hd(V_\bu,\cF)-\e.
$$
\end{thm}
\begin{proof} For each $p$ let $S^\bu V_{\le p}$ be the graded
subalgebra of $V_\bu$ generated by $V_1+...+V_p$. For each $t<e_\ma(V_\bu,\cF)$ the graded linear series $V^t_\bu$ contains an ample series by Lemma~\ref{lem:contample}, thus~\cite{LM08} Theorem 3.5 implies that
$$
\lim_{p\to\infty}\vol(S^\bu V_{\le p}^t)=\vol(V^t_\bu)
$$
for each $t<e_\ma(V_\bu,\cF)$. Since we also have
$\vol(S^\bu V_{\le p}^t)\le\vol(V^t_\bu)$ it follows from Corollary~\ref{cor:hdim} and dominated convergence that
$$
\lim_{p\to\infty}\vol \hd(S^\bu V_{\le p},\cF)=\vol\hd(V_\bu,\cF)
$$
and the result follows.
\end{proof}
This result gives in particular a simple proof of~\cite{Che08b} Theorem 5.1.

\section{Applications to the study of arithmetic volume}\label{sec:arith}
In this section $K$ will stand for a number field, $\cO_K$ for its ring of integers and $\A_K$ for its ring of ad\`eles. The set of finite (resp.~infinite) places of $K$ will be denoted by $\Sigma_f$ (resp. $\Sigma_\infty$). For each finite place $v$ we denote by $\cO_v$ the valuation ring of the completion $K_v$ of $K$.

Given a finite set $S$ we will write
$$
\hdim_K S=\frac{1}{[K:\Q]}\log\sharp S.
$$
As an (admittedly loose) justification for this notation note that
$$
\hdim_K S=(\dim_K F)\hdim_F S
$$
for any finite extension $F$ of $K$.

\subsection{Lattices}
We start by collecting a few facts on lattices. Let $v$ be a finite place of $K$, $V$ be a finite dimensional $K_v$-vector space and $\cV$ be an $\cO_v$-submodule of $V$. Recall that $\cV$ is said to be a $K_v$-\emph{lattice} of $V$ if there exists a $K_v$-basis $e_1,...,e_N$ of $V$ such that
\begin{equation}\label{equ:basis}
\cV=\cO_ve_1+...+\cO_v e_N.
\end{equation}
By~\cite{BG06} Proposition C.2.2, $\cV$ is a $K_v$-lattice iff it is finitely generated and spans $V$ over $K_v$, and also iff $\cV$ is compact and open in $V$. The map sending a norm to its unit ball establishes a one-to-one correspondence between the set of ultra-metric norms on $V$ and the set of $K_v$-lattices. A given ultra-metric norm $\|\cdot\|_v$ on $V$ is \emph{diagonalized} in a basis $e_1,...,e_N$, i.e.
$$
\Big\|\sum x_j e_j\Big\|_v=\max_j|x_j|_v
$$
holds for all $x_j\in K_v$ iff $e_1,...,e_N$ satisfies (\ref{equ:basis}) for the unit-ball $\cV$ of $\|\cdot\|_v$.

Let now $E$ be a finite dimensional $K$-vector space. An $\cO_K$-submodule $\cE$ of $E$ is said to be a $K$-\emph{lattice} iff $\cE$ is finitely generated and spans $E$ over $K$. By~\cite{BG06} Proposition C.2.6 this is the case iff its closure $\cE_v$ in the $v$-completion $E_v$ of $E$ is a $K_v$-lattice for each finite place $v$ and there exists a fixed $K$-basis $e_1,...,e_N$ of $E$ such that
$$
\cE_v=\cO_ve_1+...+\cO_v e_N
$$
holds for all but finitely many $v\in\Sigma_f$.

\subsection{Adelically normed vector spaces and their filtration by minima}

\begin{defi}\label{defi:adel} An \emph{adelically normed} $K$-vector space $\overline E$ is a finite dimensional $K$-vector space $E$ endowed for each place $v$ with a norm $\|\cdot\|_v$ on $E_v$ (with respect to $|\cdot|_v$ and ultra-metric when $v$ is finite) satisfying the \emph{local finiteness} condition:
\begin{equation}\label{equ:finite}\text{For each }x\in E\text{ we have }\|x\|_v\le 1\text{ for all but finitely many places }v.
\end{equation}
The associated $\cO_K$-module of $\overline E$ is defined
$$
\cE=\{x\in
E,\,\|x\|_v\le 1\text{ for all }v\in\Sigma_f\}
$$
and the set of \emph{small elements} is defined by
$$
\widehat{E}:=\{x\in E,\,\|x\|_v\le 1\text{ for all }v\in\Sigma\}.
$$
\end{defi}

The local finiteness condition  is equivalent to the fact that the associated $\cO_K$-submodule $\cE$ spans $E$ over $K$, and implies that for each finite place $v$ the closure $\cE_v$ of $\cE$ in $\cE_v$ spans $E_v$ over $K_v$. Since $\cE_v$ is also contained in the unit ball
$$
\{x\in E_v,\,\|x\|_v\le 1\},
$$
which is finitely generated over $\cO_v$, it follows that $\cE_v$ is a \emph{$K_v$-lattice} of $E_v$.

The \emph{algebraic envelope} $\overline E_\al$ of $\overline E$ is now defined by setting $\|\cdot\|_{\al,v}=\|\cdot\|_v$ when $v$ is infinite and letting $\|\cdot\|_{\al,v}$ be the ultra-metric norm on $E_v$ with unit ball $\cE_v$. Note that
$$
\|\cdot\|_{\al,v}\ge\|\cdot\|_v
$$
for all $v$ with equality when $v$ is infinite (but not when $v$ is finite in general). Note however that
$$
\cE_\al=\cE\text{ and }\widehat{E}_\al=\widehat{E}.
$$

\begin{defi}Let $\overline E$ be an adelically normed $K$-vector space. The induced \emph{filtration by minima} is the non-increasing $\R$-indexed filtration $\cF$ of $E$ defined by
$$
\cF_t E:=\mathrm{Vect}_K\{x\in\cE,\,\max_{v\in\Sigma_\infty}\|x\|_v\le e^{-t}\}.
$$
\end{defi}
Note that the filtrations by minima induced by $\overline E$ and its algebraic envelope $\overline E_\al$ coincide. This filtration is non-increasing and left-continuous. It is also left-bounded, but not necessarily right-bounded. Indeed we have:

\begin{prop}\label{prop:hfinite} Let $\overline E$ be an adelically normed $K$-vector space. The following conditions are equivalent.
\begin{enumerate}
\item[(i)] The associated $\cO_K$-module $\cE$ is a $K$-lattice of $E$.
\item[(ii)] The set of small elements $\widehat{E}$ is finite.
\item[(iii)] The filtration by minima is right-bounded.
\end{enumerate}
\end{prop}

\begin{proof} (ii) means that the image of $\cE$ under the diagonal embedding
$$
E\to E\otimes_\Q\R\simeq\prod_{v\in\Sigma_\infty}E_v
$$
meets the unit ball of the sup-norm
$$
\max_{v\in\Sigma_\infty}\|\cdot\|_v
$$
in a finite set. This holds iff the diagonal image of $\cE$ is a discrete subgroup, and is thus equivalent to (iii). If (i) holds then the diagonal image of $\cE$ in the $\R$-vector space $E\otimes_\Q\R$ is a lattice by~\cite{BG06} Corollary C.2.7 and therefore (ii) holds. Conversely if the diagonal image of $\cE$ is a discrete subgroup of $E\otimes_\Q\R$ then it is finitely generated over $\Z$, hence also over $\cO_K$ and the result follows since $\cE$ spans $E$ over $K$ by the local finiteness assumption in Definition~\ref{defi:adel}.
\end{proof}

\begin{rem}\label{rem:otherdef} Let us quickly compare Definition~\ref{defi:adel} to other existing notions.
\begin{itemize}
\item The local finiteness condition in Definition~\ref{defi:adel} corresponds to Assumption (A2) in~\cite{RLV00} P.3. 

\item Proposition~\ref{prop:hfinite}Ê says that $\widehat{E}$ is finite iff $\overline E_\al$ corresponds to what is called a \emph{normed vector bundle over $\Spec{\cO_K}$}, i.e. a finitely generated projective (or equivalently torsion-free) $\cO_K$-module $\cE$ together with a family of norms $\|\cdot\|_v$, $v\in\Sigma_\infty$ on $\cE\otimes_{\cO_K}K$.

\item  An \emph{adelic vector bundle over $\Spec K$} in the sense of~\cite{Gau08} D\'efinition 3.1 is an adelically normed $K$-vector space $\overline E$ in our sense which is furthermore \emph{generically trivial}: their exists a $K$-basis $e_1,...,e_N$ of $E$ in which $\|\cdot\|_v$ is diagonalized for all but finitely many $v\in\Sigma_v$. We thus see that an adelically normed vector space $\overline E$ is an adelic vector bundle iff $\widehat{E}$ is finite and $\|\cdot\|_v=\|\cdot\|_{v,\al}$ for all but finitely many $v\in\Sigma_f$. The algebraicity of adelic vector bundles is also discussed in \cite{Gau09} as {\it purity}.
\end{itemize}
\end{rem}

\subsection{Gillet-Soul\'e's Theorem}
Let $\overline E$ be an adelically normed $K$-vector space of dimension $N$. If its filtration by minima $\cF$ is right-bounded (see Proposition~\ref{prop:hfinite}) then we may consider the jumping values of
$\cF$, which will be denoted by
$$
e_1(\overline E)\ge...\ge e_N(\overline E).
$$
Note again that
$$
e_j(\overline E_\al)=e_j(\overline E).
$$
These jumping values are obtained by applying $-\log$ to the \emph{successive minima} in the sense of~\cite{BG06} Definition C.2.9.

The key point for us is now~\cite{GS91} Proposition 6, which asserts that
\begin{equation}\label{equ:gs}
\mathrm{mass}(\overline E,\cF)=\sum_{e_j(\overline E)>0}e_j(\overline E)
=\hdim_K\widehat{E}+O(N\log N)
\end{equation}
where the implicit constant in $O$ only depends on $K$. More precisely, observe first that $\overline E$ may be replaced by $\overline E_\al$ without changing both sides of (\ref{equ:gs}). The result can then be deduced from~\cite{GS91} Proposition 6 just as the adelic version of Minkowski's second theorem is deduced from its classical version in Appendix C.2.18 of~\cite{BG06} P.614.

\subsection{Adelically normed graded linear series}
\begin{defi}\label{defi:adelgraded} Let $L$ be a big line bundle over a projective $K$-variety $X$.
\begin{itemize}
\item[(i)]
An \emph{adelically normed graded linear series} $\overline{V}_\bu$ of $L$ is a graded linear series $V_\bu\subset R(L)$ such that each graded piece $V_k$ is adelically normed with norms $\|\cdot\|_{k,v}$ satisfying
$$
\| x y\|_{v,k+m}\le\|x\|_{v,k}\|y\|_{v,m}
$$
for each $x\in V_k$, $y\in V_m$ and each $v\in\Sigma$.

\item[(ii)] The \emph{arithmetic volume} of $\overline V_\bu$ is then defined by (cf. \cite{Mor09a})
$$
\hvol(\overline V_\bu):=\limsup_{k\to\infty}\frac{(n+1)!}{k^{n+1}}\hh(\overline V_k),
$$
with $n:=\dim_K X$ as usual.
\end{itemize}
\end{defi}
Our definition is a natural extension of conditions (A1) and (A2) in~\cite{RLV00} P.3. 

The filtrations by minima of the graded pieces $\overline V_k$ induce a \emph{mutliplicative filtration} $\cF$ on $V_\bu$, the \emph{filtration by minima of $\overline V_\bu$}.

This filtration is always pointwise left-bounded, but not necessarily pointwise right-bounded in general (see Proposition~\ref{prop:hfinite}).

In concrete terms the filtration by minima of $\overline V_\bu$ is linearly left-bounded iff the associated $\cO_K$-module $\cV_k$ contains a finite set of generators of $V_k$ whose norms grow at most exponentially fast with $k$, whereas it is linearly right-bounded iff the minimal norm of a non-zero vector in $\cV_k$ decays at most exponentially fast.\\

The main example of adelically normed graded linear series arises in the usual Arakelov-geometric setting. Let $\cL\to\cX$ be projective flat model of $L\to X$ over $\cO_K$ and assume that $L_\C$ is endowed with a conjugation-invariant continuous Hermitian metric. Then for each $k$ we get an algebraic adelically normed $K$-vector space $\overline{H^0(X,kL)}^{\sup}$ by taking $H^0(\cX,k\cL)$ as its associated $\cO_K$-module and the sup-norms as the norms at infinity, and this defines an adelically normed linear series $\overline{R(L)}^{\sup}$.

Denote by $e^{-\phi}$ the metric on $L_\C$ and by $h_{\cL,\phi}$ the induced (normalized) height function on $X(\overline K)$.
Recall that the \emph{essential minimum} of the height function is defined by
$$
\text{ess-min } h_{\cL,\phi}:=\sup\left\{\inf_{U(\overline K)} h_{\cL,\phi},\,U\subset X\text{ non-empty Zariski open}\right\}.
$$

\begin{lem}\label{lem:linbded} The essential minimum $\text{ess-min }h_{\cL,\phi}$ of the height function
is finite.
The filtration by minima of $\overline{R(L)}^{\sup}$ (and hence of any graded subseries) induced by the above Arakelov-geometric setting is linearly right-bounded.
\end{lem}
\begin{proof} Given a non-zero section $\sigma\in H^0(\cX,k\cL)$ and $x\in X(\overline K)$ such that $\sigma(x)\neq 0$ the height $h_{\cL,\phi}(x)$ is equal to the mean value of the function $\phi-\frac{1}{k}\log|\sigma|$ along the Galois orbit of $x$, and this easily implies that
$$
e_\ma(\overline{R(L)}^{\sup},\mathcal F)\le\text{ess-min } h_{\cL,\phi}.
$$
Hence the second assertion is a direct consequence of the first one.

In the following, we prove the first assertion.
When $\cL$ is relatively ample and $\phi$ is smooth with curvature $dd^c\phi>0$ the finiteness of $\text{ess-min }h_{\cL,\phi}$ follows from the much more precise result~\cite{Zha95} Theorem 5.2. In the general case we may first find a smooth conjugation invariant weight $\psi$ on $L_\C$ such that $\phi\le\psi$ since $\phi$ is continuous, and it follows that
$$
h_{\cL,\phi}\le h_{\cL,\psi}
$$
on $X(\overline K)$. Now let $\cA$ be a relatively ample line bundle on $\cX$ and endow $A_\C$ with a smooth conjugation invariant weight $\phi_A$ such that $dd^c\phi_A>0$. Upon replacing $(\cA,\phi_A)$ with $(m\cA,m\phi_A)$ for some $m\gg 1$ we may assume that
\begin{itemize}
\item[(i)]  There exists a non-zero section $s\in H^0(\cX,\cA)$ with sup-norm at most $1$ with respect to $\phi_A$.
\item[(ii)] $dd^c\psi+dd^c\phi_A>0$.
\item[(iii)] $\cL+\cA$ is relatively ample.
\end{itemize}
Condition (i) means that the function $\phi_A-\log|s|$ is non-negative, which implies
$$
h_{\cL,\psi}\le h_{\cL+\cA,\psi+\phi_A}
$$
hence
$$
\text{ess-min } h_{\cL,\phi}\le\text{ess-min }h_{\cL+\cA,\psi+\phi_A}.
$$
By (ii) and (iii) the right-hand side is finite by Zhang's result and we are done.
\end{proof}

\begin{defi} Let $L$ be a big line bundle on $X/K$ and let $\overline V_\bu$ be an adelically normed graded linear series of $L$ such that:
\begin{itemize}
\item[(i)] $V_\bu$ contains an ample series.
\item[(ii)] The induced filtration by minima on $V_\bu$ is linearly right-bounded.
\end{itemize}
Given the choice of a system of parameters $z=(z_1,...,z_n)$ at a regular point of $X(\overline K)$, we define the \emph{concave transform}
$$
G_{\overline V_\bu}:\Delta(V_\bu)\to[-\infty,+\infty[
$$
of $\overline V_\bu$ as the concave transform of the filtration by minima given by Definition~\ref{defi:conc}, and the \emph{arithmetic Okoukov body} is then similarly defined by
$$
\hd(\overline V_\bu)=\{(x,t)\in\Delta(V_\bu)\times\R,\,0\le t\le G_{\overline V_\bu}(x)\}.
$$
\end{defi}

As a consequence of Theorem~\ref{thm:main} we shall prove:

\begin{thm}\label{thm:volar} Let $\overline V_\bu$ be an adelically normed graded linear series which contains an ample series and whose filtration by minima is linearly right-bounded. Then we have
$$
\vol\hd(\overline V_\bu)=\lim_{k\to\infty}\frac{\hdim_K\widehat{V}_k}{k^{n+1}}=(n+1)!\hvol(\overline V_\bu).
$$
\end{thm}
\begin{proof} By Corollary~\ref{cor:hdim} we have
$$
\vol\hd(\overline V_\bu)=\lim_{k\to\infty}\frac{\mas(V_k,\cF)}{k^{n+1}}
$$
where
$$
\mas(V_k,\cF)=\sum_{e_j(V_k,\cF)>0}e_j(V_k,\cF)=\hdim_K \widehat{V}_k+O(N_k\log N_k)
$$
by Definition~\ref{defi:jump} and Gillet-Soul\'e's result (\ref{equ:gs}). Since $N_k=\dim V_k=O(k^n)$ we have
$$
N_k\log N_k=O(k^n\log k)=o(k^{n+1})
$$
and the result follows.
\end{proof}

\subsection{Fujita approximation and log-concavity}
Let us first spell out Theorem~\ref{thm:fuj} in the adelic case.

\begin{thm}\label{thm:fujar} Let $\overline V_\bu$ be an adelically normed graded linear series which contains an ample series and whose filtration by minima is linearly right-bounded. Then for every $\e>0$ there exists a finitely generated subseries $W_\bu$ of $V_\bu$ such that
$$
\hvol(\overline W_\bu)\ge\hvol(\overline V_\bu)-\e.
$$
\end{thm}
Indeed this follows directly from Theorem~\ref{thm:fuj} in view of Theorem~\ref{thm:volar}.

We next get as usual the log-concavity of arithmetic volumes.
\begin{prop}Let $L$ and $M$ be two big line bundles on $X/K$. Assume that $\overline U_\bu$, $\overline V_\bu$ and $\overline W_\bu$ are adelically normed graded linear series of $L$, $M$ and $L+M$ respectively such that each of them contains an ample series and has a linearly right-bounded filtration by minima. Assume furthermore that
\begin{itemize}
\item[(i)]  $U_k\cdot V_k\subset W_k$ for each $k$.
\item[(ii)] For any $v\in\Sigma$ and all $s\in U_k$, $s'\in V_k$, one has
$\|s\cdot s'\|_v\le\|s\|_v\|s'\|_v$.
\end{itemize}

Then one has
\begin{equation}\label{Equ:loc-concavite}\hvol(\overline W_\bu)^{\frac{1}{n+1}}\ge
\hvol(\overline U_\bu)^{\frac1{n+1}}+\hvol(\overline V_\bu)^{\frac 1{n+1}}.
\end{equation}
\end{prop}

\begin{proof}
The assumptions easily imply that
$$
\cF_s U_k\cdot\cF_t V_k\subset\cF_{s+t} W_k
$$
for all $k$ and all $s,t\in\R$, which yields in turn
$$
\hd(\overline U_\bu)+\hd(\overline V_\bu)\subset\hd(\overline W_\bu)
$$
and the result follows by the Brunn-Minkowski inequality.
\end{proof}

\subsection{Convergence of the Euler-Poincar\'e characteristic}
Given an adelically normed $N$-dimensional $K$-vector space $\overline E$ we consider as in~\cite{RLV00,Gau08,Yua09} its (normalized) \emph{adelic Euler characteristic}
\begin{equation}\label{equ:chi0}
\chi(\overline E):=\frac{1}{[K:\Q]}\log\vol \B(\overline E)\in]-\infty,+\infty]
\end{equation}
where
$\B(\overline E)\subset E\otimes_K\A_K$ denotes the adelic unit ball induced by the family of norms $\|\cdot\|_v$ of $\overline E$ and $\vol$ is the Haar measure of $E\otimes_K\A_K$ normalized by
$$
\vol(E\otimes_K\A_K/E)=1
$$
(compare this normalization to~\cite{BG06} Proposition C.1.10). If $\overline E$ is \emph{generically trivial} (i.e.~an adelic vector bundle in the sense of~\cite{Gau08}, cf.~Remark~\ref{rem:otherdef}) then
$\chi(\overline E)$ is finite and we have by definition the Riemann-Roch-type formula
$$
\chi(\overline E)=\hdeg_n(\overline E)+\chi(\overline{K^N})
$$
where $\hdeg_n$ denotes the \emph{normalized adelic degree} of~\cite{Gau08} D\'efinition 4.1 and $\overline{K^N}$ is adelically normed in the trivial ($\ell^2$) way. A computation shows that $\chi(\overline{K^N})=O(N\log N)$, so that $\chi$ and $\hdeg_n$ coincide modulo $O(N\log N)$.

Note that $\chi(\overline E_\al)\le\chi(\overline E)$ but equality doesn't hold in general. In fact $\overline E_\al$ is locally trivial when $\widehat{E}$ if finite by Proposition~\ref{prop:hfinite}. We thus have $\chi(\overline E_\al)<+\infty$ in that case, but it might still be the case that $\chi(\overline E)=+\infty$.

When $\widehat{E}$ is finite the adelic version of Minkowski's second theorem applies to $\overline E_\al$ (\cite{BG06} Theorem C.2.11) and yields
\begin{equation}\label{equ:mink}
\chi(\overline E_\al)=\sum_{j=1}^Ne_j(\overline E)+O(N\log N)
\end{equation}
where the constant in $O$ only depends on $K$ (recall that $e_j(\overline E_\al)=e_j(\overline E)$). As a consequence we will show similarly to Theorem~\ref{thm:volar}:

\begin{thm}\label{thm:chi} Let $\overline V_\bu$ be as in Theorem~\ref{thm:volar} an adelically normed graded linear series of $L$ which contains an ample series and whose filtration by minima is linearly right-bounded. If the filtration by minima is furthermore linearly left-bounded (e.g. if $V_\bu$ is finitely generated) then we have
$$\int_{\Delta(V_\bu)}G_{\overline V_\bu}d\lambda=\lim_{k\to\infty}\frac{\chi(\overline V_{k,\al})}{k^{n+1}}
$$
\end{thm}
\begin{proof} Set as in Theorem~\ref{thm:main}
$$
\mu_k:=k^{-n}\sum_j\delta_{k^{-1}e_j(V_k,\cF)}.
$$
By Minkowski's second theorem (\ref{equ:mink}) we have
$$
\frac{\chi(\overline V_{k,\al})}{k^{n+1}}=\int_\R t \mu_k(dt)+o(1).
$$
On the other hand since the filtration by minima is assumed to be both left and right-bounded the support of $\mu_k$ stays in a fixed compact set independent of $k$ and Theorem~\ref{thm:main} therefore yields
$$
\lim_{k\to\infty}\int_\R t \mu_k(dt)=\int_{\Delta(V_\bu)}G_{\overline V_\bu}d\lambda,
$$
which proves the result.
\end{proof}

\section{Application to the sectional capacity}\label{sec:seccapa}

In this section, we shall associate to each adelically normed vector space another tyoe of filtration, the \emph{filtration by height} and use it to study the existence of sectional capacity for adelically normed graded linear series. Again $K$ denotes a number field.

\subsection{Filtration by height}
As in~\cite{Gau08} D\'efinition 4.11 we introduce:
\begin{defi}
Let $\overline E$ be an adelically normed vector space. The \emph{height} of a non-zero element $x\in E$ is defined by
$$
h_{\overline E}(x):=\sum_{v}\frac{[K_v:\mathbb Q_v]}{[K:\mathbb Q]}\log\|x\|_v\in[-\infty,+\infty[
$$
(recall that $\|x\|_v\le 1$ for all but finitely many $v$'s).
The \emph{filtration by height} is defined as
$$
\widetilde{\cF}_tE:=\mathrm{Vect}_K\{x\in E,\, h_{\overline E}(x)\le -t\}.
$$
\end{defi}

Note that the filtration by height is always left-bounded (given a basis $(x_1,...,x_N)$ of $E$ we have $\widetilde{\cF}_tE=E$ for $t\ge -\min_j h_{\overline E}(x_j)$). It is right-bounded if $\overline E$ is generically trivial (i.e.~an adelic vector bundle over $\Spec K$). This is a consequence of Northcott's theorem in the adelic vector bundle setting, see \cite{Gau08} Proposition 4.12. 

By~\cite{Gau08} Lemme 4.4 we have $h_{\overline{E}}(x)=-\hdeg_n(Kx)$, and it follows that $e_{\max}(E,\widetilde{\mathcal F})$ coincides with $\mathrm{u\widehat{deg}}_n(\overline E)$, the (normalized) upper degree of $\overline E$, introduced and studied in \cite{BK07}.

Using Definition \ref{defi:jump} it is easy to check that
\begin{equation}\label{equ:inf}
\sum_{j=1}^N e_j(E,\widetilde{\mathcal F})=-\inf_{(x_j)}\sum_{j}h_{\overline E}(x_j)\in\mathbb R\cup\{-\infty\}, 
\end{equation}
where $(x_j)_{j=1}^N$ runs over the set of all bases of $E$.

By Proposition 4.13 and Theorem 4.14 of~\cite{Gau08} (a generalized Siegel lemma) for any \emph{generically trivial} adelically normed vector space $\overline E$ we have
\begin{equation}\label{equ:siegel}
\chi(\overline E)=\sum_{j=1}^Ne_j(E,\widetilde{\cF})+O(N\log N)
\end{equation}
where the constant in $O$ only depends on $K$.

\begin{defi}
Let $\overline V_{\bu}$ be an adelically normed graded linear series of $L$ such that:
\begin{itemize}
\item[(i)] $V_{\bu}$ contains an ample series.
\item[(ii)] Each graded piece $\overline V_k$ is generically trivial, and the filtration by height of $V_\bu$ is furthermore linearly right-bounded. 
\end{itemize}
Given a system of parameters $z$ at a regular point of $X(\overline K)$, we define 
$$
\widetilde{G}_{\overline{V}_\bu}:\Delta(V_\bu)\rightarrow[-\infty,+\infty[
$$
as the concave transform of the filtration by height.
\end{defi}

Similarly to Theorem \ref{thm:chi}, we express the sectional capacity of $\overline{V}_\bu$ as the integral of the function $\widetilde{G}_{\overline V_\bu}$.

\begin{thm}\label{thm:e-p}
Let $\overline{V}_\bu$ be an adelically normed graded linear series of $L$ which contains an ample series and such that each graded piece $\overline V_k$ is generically trivial. Assume that the filtration by heights of $\overline V_\bu$ is linearly left and right bounded. Then we have
\[\int_{\Delta(V_\bu)}\widetilde G_{\overline V_\bu}\,d\lambda
=\lim_{k\rightarrow\infty}\frac{\chi(\overline V_{k})}{k^{n+1}}.\]
\end{thm}
\begin{proof}
The proof follows the same line as that of Theorem \ref{thm:chi} in replacing the Minkowski's second theorem by (\ref{equ:siegel}). 
\end{proof}

\begin{rem}Let $\overline V_\bu$ be an adelically normed graded linear series of $L$.
\begin{itemize}
\item[(i)] As shown by Lemma \ref{lem:fingen}, the filtration by heights of $\overline V_\bu$ is linearly left-bounded provided that it is pointwise left-bounded and that the algebra $V_\bu$ is finitely generated. This condition is strongly similar to the increasing speed condition figuring in the Theorem (A) of \cite{RLV00}.
\item[(ii)] In the Arakelov-geometric setting, the filtration by height is automatically linearly right-bounded. In fact, one has \[e_{\max}(\overline V_\bu,\widetilde{\mathcal F})\le e_{\max}(\overline{R(L)}^{\sup},\widetilde{\mathcal F}),\] where the latter is bounded from above by the essential minimum of the height function on $X$ (which is finite by Lemma~\ref{lem:linbded}). 
\end{itemize}
\end{rem}

\section{Comparison with other results}\label{sec:yuan}

\subsection{Asymptotic measures}
We first compare our results with some of the first-named author's previous results. In~\cite{Che07,Che08b} the author constructed the \emph{asymptotic measure} of a big line bundle $L$ endowed with Arakelov-geometric data $\overline L$ as above. This measure describes the asymptotic distribution of the jumping values of the \emph{Harder-Narasimhan filtration} (see~\cite{Che07} Section 2.2). Now a comparison of the latter filtration with the filtration by minima considered above as in~\cite{Che08b} shows that the asymptotic measure coincides with the limit measure we obtain in Theorem~\ref{thm:main}, i.e.~with the direct images of $\lambda$ (by $G_{\overline V_\bu}$ and by $\widetilde G_{\overline V_\bu}$). A special case of this general phenomenon appears in the example computed in~\cite{Che07}  Section 4.1.5.

\subsection{Sectional capacity}Ê
We now relate Theorem~\ref{thm:main}  to Lau-Rumely-Varley's results~\cite{RLV00} on the  \emph{sectional capacity}.

Given an arbitrary adelically normed graded linear series $\overline V_\bu$ the quantity
$$
\exp\left(-(n+1)!\lim_{k\to\infty}\frac{\chi(\overline V_k)}{k^{n+1}}\right)
$$
is called the \emph{sectional capacity} in~\cite{RLV00}, provided the limit exists, and indeed the main result of~\cite{RLV00}Ê shows the existence of this limit when $V_\bu=R(L)$ with $L$ ample. Theorem~\ref{thm:e-p} therefore yields in paerticular a new (and much shorter) proof of the existence of the sectional capacity in the Arakelov-geometric setting when $L$ is ample on $X$ (since $R(L)$ is finitely generated in that case).

\subsection{Yuan's construction}
As mentioned in the introduction, given a big line bundle $L$ endowed with Arakelov-geometric data $\overline L$ Yuan constructs in his recent paper~\cite{Yua09} a concave function $-c[\overline L]$ on the Okounkov body $\Delta(L)$ whose mean value (under adequate assumptions) computes the sectional capacity. His construction consists in summing up the analogues of Witt-Nystr\"om's Chebyshev-type transforms over all places of $K$. We would like here to give an alternative description of his construction. We will then show in the next section that Yuan's function does \emph{not} coincide in general with our concave transform.

Let $L$ be a big line bundle on $X/K$ and let $\overline V_\bu$ be an adelically normed graded linear series of $L$. Assume furthermore that each adelically normed $K$-vector space $\overline V_k$ is \emph{locally trivial} (cf.~Remark~\ref{rem:otherdef}), so that its adelic Euler characteristic $\chi(\overline V_k)$ is finite.

Given a system of parameters $z=(z_1,...,z_n)$ at a given regular point $p\in X(\overline K)$ the valuation $\ord_z$ induces an $\N^n$-indexed filtration of each $V_k$. Let $\cG_{k,\a}$, $\a\in\N^n$, be the associated graded pieces, i.e.
$$
\cG_{k,\a}:=\{s\in V_k,\,\ord_z(s)\ge\a\}/\{s\in V_k,\,\ord_z(s)>\a\},
$$
and endow them with the induced adelically normed vector space structure $\overline{\cG}_{k,\a}$. Each $\cG_{k,\a}$ is at most one-dimensional by the basic property (\ref{equ:graded}) of the valuation $\ord_z$, and we have $\cG_{k,\a}\neq 0$ iff $(\a,k)\in\Gamma(V_\bu)$. Let $h:\Gamma(V_\bu)\to\R$ be defined by
\begin{equation}\label{equ:h}
h(k,\a):=\hdeg_n(\overline{\cG}_{k,\a}),
\end{equation}
where $\hdeg_n$ denotes as before Gaudron's normalized adelic degree (\cite{Gau08} D\'efinition 4.1). Since $\cG_{k,\a}$ is one-dimensional we have the usual description
\begin{equation}\label{equ:hdeg}
\hdeg_n(\overline{\cG}_{k,\a})=\frac{1}{[K:\Q]}\sum_v\log\|s\|_v^{-1}
\end{equation}
for any non-zero $s\in\cG_{k,\a}$ (\cite{Gau08} Lemme 4.4). One easily infers from this that $h$ is super-additive, and we have
\begin{equation}\label{equ:chi}
\chi(\overline V_k)=\sum_{\a\in\Gamma_k}h(k,\a)+O(k^n\log k)
\end{equation}
with $\Gamma_k:=\Gamma(V_\bu)\cap(\{k\}\times\N^n)$. Indeed this follows from~\cite{Gau08} Proposition 4.22, which shows that $\hdeg_n$ is additive in exact sequences modulo $O(N\log N)$, using also $\hdeg_n=\chi+O(N\log N)$.

Now assume that $V_\bu$ contains an ample series, so that the semigroup $\Gamma(V_\bu)$ spans $\Z^{n+1}$ as a group. Assume also that there exists $C>0$ such that
\begin{equation}\label{equ:bound}
-Ck\le h(k,\a)\le C(k+|\a|)
\end{equation}
for all $(k,\a)\in\Gamma(V_\bu)$. By~\cite{WN09} Section 3 the upper bound in (\ref{equ:bound}) enables to consider the concave envelope
$$
H_{\overline V_\bu}:\Delta(V_\bu)\to\R
$$
of the super-additive function $h$ (compare Remark~\ref{rem:wn} above), which satisfies
\begin{equation}\label{equ:env}
H_{\overline V_\bu}(\a)=\lim_{k\to\infty}\frac{1}{k}h(k,\a_k)
\end{equation}
for any sequence $(k,\a_k)\in\Gamma(V_\bu)$ such that $\frac{\a_k}{k}\to\a\in\Delta(V_\bu)^\circ$, and is in particular bounded above and below by (\ref{equ:bound}). One then shows exactly as in~\cite{WN09}  Section 8.2 that (\ref{equ:chi}) implies
\begin{equation}\label{equ:int}
\int_{\Delta(V_\bu)}H_{\overline V_\bu}d\lambda=\lim_{k\to\infty}\frac{\chi(\overline V_k)}{k^{n+1}}.
\end{equation}

Let us now assume that $V_\bu=R(L)$ is endowed with the adelic norms $\overline{R(L)}^{\sup}$ coming from the Arakelov-geometric setting. This is essentially the setting considered in~\cite{Yua09}. In that case Yuan's $F[m\overline L]$ satisfies
$$
F[m\overline L](\a)=-h(m,\a)
$$
for any $(m,\a)\in\Gamma(L)$, as one easily sees using (\ref{equ:hdeg}). The upper bound in (\ref{equ:bound}) is then shown to hold true in~\cite{Yua09} Lemma 2.3, and is the counterpart in this setting of Lemma~\ref{lem:linbded}. On the other hand, the lower bound is \emph{assumed} to be true in~\cite{Yua09} Theorem 1.2, and (\ref{equ:int}) is then equivalent to Yuan's result, since
$$
c[\overline L]=-H_{\overline V_\bu}.
$$
Yuan conjectures that this lower bound always holds (in the Arakelov-geometric setting), and notes that it is the case when $(X,L)=(\PP^n,\cO(1))$ (\cite{Yua09} P.18). We observe here that the lower bound holds more generally as soon as $V_\bu$ is finitely generated (a counterpart in this setting of Lemma~\ref{lem:fingen}). In particular Yuan's conjecture is true when $L$ is semiample.

\subsection{A counterexample}
In view of Theorem~\ref{thm:chi} and Yuan's result (\ref{equ:int}), it is certainly tempting to imagine that Yuan's function coincides with our concave transform. This is however not the case as we shall now see.

Let $K:=\Q$, $X:=\PP^1$ and $L:=\cO(1)$, endowed with the Arakelov-geometric data $\overline L$ given by the standard model $(\cX,\cL)$ of $(X,L)$ over $\Z$ and the Fubiny-Study metric on $L_\C$.

Given a point $p\in X(\Q)$ we may then consider the valuation $\ord_p$, which doesn't depend on the specific choice of a coordinate in this one-dimensional situation. The Okounkov body $\Delta=\Delta(L)$ of $L$ is then equal to the unit-segment $[0,1]\subset\R$ for any choice of $p$ (cf.~\cite{LM08} Example 1.13).

Since $\overline L$ is arithmetically ample, we have on the one hand $e_\mi(\overline L)\ge 0$, hence $G_{\overline L}\ge 0$ by (\ref{equ:bounds}).

On the other hand pick $m\in\Z$ and choose $p=[m:1]$ in homogeneous coordinates. Under the standard identification $H^0(\cX,k\cL)=\Z[X,Y]_k$ with homogeneous polynomials of degree $k$ we have
$$
\{s\in H^0(kL),\,\ord_p(s)\ge\a\}=(X-mY)^\a\Z[X,Y]_{k-\a}
$$
hence the function $h:\Gamma(\overline L)\to\R$ defined by (\ref{equ:h}) satisfies
$$
h(k,k)=\deg_n(\overline{\Q\cdot(X-mY)^k})=-\log\|(X-mY)^k\|_{\sup}
$$
$$
=-\log\sup_{(x,y)\neq(0,0)}\frac{(x-my)^k}{(x^2+y^2)^{k/2}}=-\frac{k}{2}\log(1+m^2).
$$
By (\ref{equ:env}) Yuan's function $H_{\overline L}:[0,1]\to\R$ satisfies
$$
H_{\overline L}(1)=-\frac{1}{2}\log(1+m^2).
$$
We thus conclude that the functions corresponding to any point $p:=[m:1]$ with $m\neq 0$ satisfy
$$
\inf_{\Delta(L)} H_{\overline L}<0\le\inf_{\Delta(L)} G_{\overline L}
$$
which shows indeed that they cannot be equal.

\subsection{The toric case}
Let us briefly discuss the toric case. Let $X$ be a smooth toric variery and $L$ be a toric big line bundle with associated polytope $\Delta$. Note that $(X,L)$ is automatically defined over $\Z$. The toric line bundle $L_\C$ is canonically trivialized on the torus $(\C^*)^n\subset X(\C)$, and a continuous invariant metric on $L$ therefore defines a function on $(\C^*)^n$ of the form
$g(\log|z_1|,...,\log|z_n|)$. The Legendre transform
$$
g^*(t):=\sup_{s\in\R^n}\left(\langle s,t\rangle-g(s)\right)
$$
is a continuous convex function on $\Delta$ (known as the \emph{symplectic potential} when $\phi$ is smooth and positively curved). If one uses a toric system of parameters $z$ at a toric point $p$ of $X$ then the Okounkov body of $L$ coincides with the polytope $\Delta$ (cf.~\cite{LM08} Section 6.1), and one can then check in this specific case that both our concave transform $G_{\overline L}$ and Yuan's function $H_{\overline L}$ coincide with $-g^*$. Note that a similar construction appears also in the recent work of Burgos, Philippon and Sombra \cite{BPS09}.

We are not going to prove this, but only indicate the key points. The first idea is to replace as in~\cite{WN09}Ê Section 9.3 the sup-norm by the $L^2$-norm with respect to a volume form invariant under the compact torus. These norms are not sub-multiplicative anymore, but one can still consider the functions $g$ and $h$ defined by (\ref{equ:g}) and (\ref{equ:h}), whose concave envelope will still compute $G_{\overline L}$ and $H_{\overline L}$ by the usual argument that the distortion between the sup-norm and the $L^2$-norm has subexponential growth. Now the decomposition
$$
H^0(kL)=\bigoplus_{\a\in k\Delta}\Q s_\a
$$
in monomials $s_\a$ is both orthogonal with respect to the $L^2$-scalar product and defined over $\Z$, and this enables to perform the computation of the functions $g$ and $h$ explicitely. One then concludes exactly as in~\cite{WN09} Lemma 9.2.

\end{document}